\documentclass [12pt,openbib]{article}
\usepackage{lipsum}

\newcommand\blfootnote[1]{%
  \begingroup
  \renewcommand\thefootnote{}\footnote{#1}%
  \addtocounter{footnote}{-1}%
  \endgroup
}

\usepackage{pslatex}
\usepackage{graphicx}
\usepackage{amsmath,,amssymb,amsfonts,amsthm}
\usepackage{xcolor}
\usepackage[utf8]{inputenc}
\usepackage[multiple]{footmisc}
\usepackage{setspace}
\usepackage{MnSymbol,wasysym}
\usepackage[export]{adjustbox}
\newtheorem{theorem}{Theorem}

\newtheorem{lemma}{Lemma}
\usepackage{enumitem}
\usepackage{color}
\usepackage{fancyhdr}
\usepackage{esint}
\usepackage{mathtools}
\usepackage[margin=1in]{geometry}
\newenvironment{narrow}[2]{%
 \begin{list}{}{%
  \setlength{\topsep}{0pt}%
  \setlength{\leftmargin}{#1}%
  \setlength{\rightmargin}{#2}%
  \setlength{\listparindent}{\parindent}%
  \setlength{\itemindent}{\parindent}%
  \setlength{\parsep}{\parskip}%
 }%
\item[]}{\end{list}}
\def\begeq{\begin{equation}}
\def\endeq{\end{equation}}
\def\bbr{\begin{bmatrix*}[r]}
\def\ebr{\end{bmatrix*}}
\def\bb{\begin{bmatrix*}[r]}
\def\eb{\end{bmatrix*}}
\def\bbc{\begin{bmatrix}}
\def\ebc{\end{bmatrix}}

\def\Re{{\hbox{\rm I\kern-.2150em R}}}
\definecolor{darkmagenta}{rgb}{0.55, 0.0, 0.55}

\definecolor{darkgreen}{rgb}{0.0,0.6,0.0}

\def\Re{\mathbb{R}}

\def\R{\mathbb{R}}

\def\calL{\mathcal L}

\def\0u{\underline{0}}

\def\begeq{\begin{equation}} 
\def\endeq{\end{equation}}

\def\bcenter{\begin{center}}
\def\ecenter{\end{center}}
\def\beq{\begin{equation}}
\def\eeq{\end{equation}}
\def\bmatr{\begin{bmatrix*}[r]}
\def\ematr{\end{bmatrix*}}
\def\bmatc{\begin{bmatrix}}
\def\ematc{\end{bmatrix}}
\usepackage{amsmath}               
  {
      \theoremstyle{plain}
      
  }
   {
      \theoremstyle{plain}
      
  }
   {
      \theoremstyle{plain}
      
  }
  {
      \theoremstyle{plain}
      
  }
\newcommand{\vertiii}[1]{{\left\vert\kern-0.25ex\left\vert\kern-0.25ex\left\vert #1 
    \right\vert\kern-0.25ex\right\vert\kern-0.25ex\right\vert}}
 \usepackage{amsthm}
\newtheorem*{proposition*}{Proposition}
\newtheorem*{definition*}{Definition}
\newtheorem*{example*}{Example}
\newtheorem*{lemma*}{Lemma}
\newtheorem*{theorem*}{Theorem}
\newtheorem*{corollary*}{Corollary}
\newtheorem*{remark*}{Remark}

\begin{document}

\begin{center}
{\Large Slow Convergence in Generalized \\ \vskip 3pt Central Limit Theorems}
\vskip 10pt
{\Large \em Convergence Lente dans les Théorèmes  \\ \vskip 3pt Central Limite Généralisés}
\vskip 10pt
Christoph B\"orgers\footnote{Department of Mathematics, Tufts University, Medford, MA, 02155, United States.}$^{,}$\footnote{Email address: christoph.borgers@tufts.edu.} and Claude Greengard\footnote{Visiting Scholar, Courant Institute of Mathematical Sciences, New York University and Foss Hill Partners, P.O.\ Box 938, Chappaqua, NY 10514, United States.}$^{,}$\footnote{Email address: cg3197@nyu.edu.}
\vskip 10pt
\end{center}

\vskip 15pt

\begin{center}
{\bf Abstract}\blfootnote{\textcopyright 2018. This manuscript version is made available under the CC-BY-NC-ND 4.0 license\newline https://creativecommons.org/licenses/by-nc-nd/4.0/}
\end{center}

\begin{narrow}{0.41 in}{0.41 in}
We study the central limit theorem in the non-normal domain of attraction to symmetric $\alpha$-stable laws for $0<\alpha\leq2$. We show that for i.i.d.\ random variables $X_i$, the convergence rate in  $L^\infty$ of both the densities and distributions of $\sum_i^n X_i/(n^{1/\alpha}L(n))$ is at best logarithmic if $L$ is a non-trivial slowly varying function. Asymptotic laws
for several physical processes have been derived using convergence of $\sum_{i=1}^nX_i/\sqrt{n\log n}$ to Gaussian distributions.
Our result implies that such asymptotic laws are accurate only for
exponentially large $n$.
\end{narrow}

\begin{center}
{\bf Résumé}
\end{center}

\begin{narrow}{0.41 in}{0.41 in}
Nous étudions le théorème central limite dans le domaine d’attraction non-normal, vers des limites symétriques et $\alpha$-stables,  $0<\alpha\leq2$. Nous montrons que pour les suites $X_i$ i.i.d., les taux de convergence en $L^\infty$ des densités et des distributions de $\sum_i^n X_i/(n^{1/\alpha}L(n))$ sont au plus logarithmiques si $L$ est une fonction non-triviale de variation lente. Plusieurs lois physiques asymptotiques sont basées sur la convergence des suites $\sum_{i=1}^nX_i/\sqrt{n\log n}$ vers des distributions Gaussiennes. Nos résultats montrent que ces lois ne sont précises que pour $n$ d’une grandeur exponentielle.
\end{narrow}

\vskip 15pt

\begin{center}
{\bf Introduction}
\end{center}

In the kinetic theory of gases and related areas, there are a number
of limit laws that involve logarithmic scaling. Our interest in the
subject was initiated by B\"orgers et al.~\cite{borgers1992diffusion},
where we studied the evolution of free molecular flow in a region
bounded by parallel plates with Maxwellian reflection on the boundaries.
Using probabilistic methods, we showed that in the limit of vanishing
gap height, diffusion occurs on the ``anomalous" time scale $1/(h|\log h |)$,
where $h$ is the properly non-dimensionalized separation of the plates.
More recently, Chumley et al \cite{chumley2016diffusivity} have
obtained diffusion results, both standard and anomalous, for related
kinetic problems in a very broad class of geometries. Anomalous diffusion
results involving logarithmic scaling have also been obtained for
Birkhoff sums in the stadium billiard problem \cite{balint2006limit},
for diffusion in a Lorentz gas \cite{dettmann2014diffusion}, for
iterations of certain one-dimensional maps \cite[page 88, remark 2]{gouezel2004central},
and in two-dimensional turbulence models \cite{jimenez1996algebraic}.

In the research cited above, at the core of the results are central
limit theorems for random variables of infinite variance. The limiting
distributions are Gaussian, as in the classical central limit theorem,
but the infinite variance of the summands necessitates scaling by
$\sqrt{n\log n}$ instead of 
$\sqrt{n}$. 
The summands are independent
and identically distributed, or, in the deterministic problems, have
asymptotically vanishing correlations. The log term in the diffusion
time scale is a consequence of the logarithmic scaling in the central
limit theorem. 

Here we consider, more generally, central limit theorems for random
variables with infinite variance in which the limiting distribution
is symmetric and $\alpha$-stable, with $0<\alpha\leq2$ \cite{ChristophWolf,Zolo}; see \cite{nolan2010bibliography}
for an extensive bibliography on stable distributions and applications.
We consider summands outside the normal domain of attraction; that
is, we assume scaling not by $n^{1/\alpha}$ but by $n^{1/\alpha}L(n)$,
where $L$ is a slowly varying function 
such that $L(n)$ tends to $0$ or to $\infty$ as $n\rightarrow\infty$.
In the examples discussed earlier, $\alpha=2$ and $L(n)=\sqrt{\log n}$.

A natural question is how long must one wait in order for the asymptotic
result to provide a good approximation of the physical phenomenon
being modeled. 
 We
will not attempt to review the considerable body of literature concerning
rates of convergence \textit{inside} the normal domain of attraction
here.
In this paper, we use a scaling argument to prove that convergence
of the distribution and density functions in  $L^{\infty}$  is always slow {\em outside} the normal domain of attraction.
We show that the rate of convergence cannot be better than inverse
logarithmic. Hence any approximation based on a central limit theorem
outside the normal domain of attraction of a stable distribution will
be accurate only for an exponentially large number of summands.

The rates of convergence in central limit theorems depend critically
on the shapes of the distributions and the corresponding scalings
of the partial sums. The main results in this paper provide general best-case
bounds on the rates of convergence. There are several results providing
{\em exact} rates of convergence for special families of distributions and
scalings outside
the normal domain of attraction; see for instance Juozulynas
and Paulauskas \cite{juozulynas1998some}, Kuske and Keller \cite{keller2001rate},
and N\'andori \cite{nandori2011recurrence}. We conclude this paper with an example of a family of random variables which have an $O(\log(\log n))/\log n$  rate of convergence under a natural scaling and the improved $O(1/\log n)$ rate of convergence under an alternative scaling. 

We note  that 
Cristadoro et al \cite{cristadoro2014measuring} gave an analysis, 
supported by numerical simulations, of anomalous diffusion in the 
Lorentz gas, and found slow convergence to the
limiting distribution.

\vskip 15pt
\begin{center}
{\bf Results}
\end{center}

The two theorems in this paper demonstrate that rates of convergence outside
the normal domain of attraction, with scaling  $n^{1/\alpha}L(n)$,
are at best of order 
$1-L(n)/L(2n)$.
The proposition following the theorems makes clear why we call this convergence
slow and ``at best logarithmic.''

The first theorem provides a best-case bound on the rate of convergence of the error as measured by the Kolmogorov metric  \cite{Rachevetal}, $d$, which is defined as follows. Let $X,Y$ be random variables, with corresponding distribution functions $F_X,F_Y$. The Kolmogorov distance between $X$ and $Y$ is 
$$
d(X,Y) = \| F_X - F_Y \|_\infty.
$$

The proof of Theorem \ref{thm:firstone} relies on subadditivity properties of the Kolmogorov distance \cite[Section 4.2]{boutsikas2001compound}. We recall these properties in the following two lemmas.

\begin{lemma} 
\label{lemma:shift_Kolmogorov} 
Let $X$, $Y$, and $Z$ be random variables, and assume that $Z$ is independent of
$X$ and of $Y$. Then 
$$
d (X+Z,Y+Z) \leq d(X,Y).
$$
\end{lemma}

\begin{proof} 
The distribution function of $X+Z$ is given by the Lebesgue-Stieltjes integral
$$
F_{X+Z}(s) = 
\int_{-\infty}^{\infty} F_X(s-z) dF_Z(z), 
$$
and similarly
$$
F_{Y+Z}(s) = 
\int_{-\infty}^{\infty} F_Y(s-z) dF_Z(z). 
$$
Hence,
\begin{eqnarray*}
d(X+Z,Y+Z) &=&  \sup_{s \in \R} \left| F_{X+Z}(s) - F_{Y+Z}(s) \right|  \\ 
&=& \sup_{s \in \R} \left| \int_{-\infty}^{\infty} F_X(s-z) ~\! dF_Z(z)  - 
\int_{-\infty}^{\infty} F_Y(s-z) ~\! dF_Z(z) \right|  \\ 
&\leq&  \sup_{s \in \R} \int_{-\infty}^{\infty}  \left| F_X(s-z) - F_Y(s-z) \right| dF_Z(z)  \\ 
&\leq&   \int_{-\infty}^{\infty} \sup_{s \in \R} \left| F_X(s-z) - F_Y(s-z) \right| dF_Z(z)  \\ 
&=& \int_{-\infty}^\infty d(X,Y)~\!  dF_Z(z) = d(X,Y).
\end{eqnarray*}
\end{proof}

\begin{lemma} 
\label{secondlemma}
Let $(X_1,Y_1)$ and $(X_2,Y_2)$ be independent pairs of random variables.
Then
$$
d(X_1+X_2,Y_1+Y_2) \leq d(X_1,Y_1) + d(X_2,Y_2).
$$
\end{lemma}

\begin{proof}
By the triangle inequality and Lemma \ref{lemma:shift_Kolmogorov}, we have 
\begin{eqnarray*}
d(X_1+X_2,Y_1+Y_2)  &\leq& d(X_1+X_2,Y_1+X_2) + d(Y_1+X_2,Y_1+Y_2)  \\
&\leq& d(X_1,Y_1) + d(X_2,Y_2).
\end{eqnarray*}
\end{proof}

Before stating and proving our theorems, we introduce some notation
that we will use in the rest of this paper. When $G$ is the distribution
function of a random variable $X$ and $a\in\mathbb{R}$, we denote
by $T_aG$ the distribution function of the scaled random variable
$aX$, so
\[
T_aG(x)=G(x/a).
\]
 Similarly, if $q$ is the density function of $X$, we denote
by $\tau_aq$ the density function of $aX$, so 
\[
\tau_aq(x)=\left(1/a\right)q(x/a).
\]

\begin{theorem}
\label{thm:firstone}Let $X_{1},X_{2},\ldots$ be a sequence of independent,
identically distributed, random variables. Denote by $S_{n}$ the sum $S_{n}=\sum_{i=1}^{n}X_{i}$.
Assume that for some $\alpha,$ $0< \alpha \leq 2,$ 
\[
\frac{S_{n}}{n^{1/\alpha}L(n)}\Longrightarrow Z,
\]
where $Z$ is symmetric and $\alpha$-stable with distribution function
$F$, $L$ is a slowly varying function, and the symbol $\Rightarrow$ denotes convergence in distribution.
Denote by $F_{n}$ the distribution function
\[
F_{n}(x)=P\left(\frac{S_{n}}{n^{1/\alpha}L(n)}\leq x\right).
\]
Then for all $C<1$ and  $z \in \R$,
$$
2\|F_{n}-F\|_{\infty}+\|F_{2n}-F\|_{\infty}\geq C\left|zF^{\prime}(z)\right| \left|1-\frac{L(n)}{L(2n)}\right|.
$$
Hence, provided $L(n)\neq L(2n)$ for all sufficiently large $n$,
\[
\limsup_{n \rightarrow \infty} \frac{\| F_n - F \|_\infty}{|1-L(n)/L(2n)|} > 0 ~~\mbox{or} ~~
\limsup_{n \rightarrow \infty} \frac{\| F_{2n} - F \|_\infty}{|1-L(n)/L(2n)|} > 0. 
\] 
\end{theorem}
\begin{proof}
Define $\tilde{S}_{n}=\sum_{i=n+1}^{2n}X_{i}$ and let $Z_{1},Z_{2}$
be independent copies of $Z.$ From Lemma \ref{secondlemma} and the $\alpha-$stability and symmetry
of $Z$, we
have that
\begin{align}
\nonumber
\left\Vert T_{L(2n)/L(n)}F_{2n}-F\right\Vert _{\infty} & =d\left(\frac{L(2n)}{L(n)}\frac{S_{2n}}{\left(2n\right)^{1/\alpha}L(2n)},Z\right)\\
\nonumber
 & =d\left(\frac{S_{n}}{\left(2n\right)^{1/\alpha}L(n)}+\frac{\tilde{S}_{n}}{\left(2n\right)^{1/\alpha}L(n)},\frac{Z_{1}}{2^{1/\alpha}}+\frac{Z_{2}}{2^{1/\alpha}}\right)\\
 \nonumber
 & \leq2d\left(\frac{S_{n}}{\left(2n\right)^{1/\alpha}L(n)},\frac{Z}{2^{1/\alpha}}\right)\\
 \nonumber
 & =2d\left(\frac{S_{n}}{n^{1/\alpha}L(n)},Z\right)\\
 & =2\left\Vert F_{n}-F\right\Vert _{\infty}.
 \label{one}
\end{align}
On the  other hand, 
\begin{equation}
\|{}T_{L(2n)/L(n)}F_{2n}~ - ~~ T_{L(2n)/L(n)}F\|_{\infty}=\|F_{2n}-F\|_{\infty}.
\label{two}
\end{equation}

Let $C<1$ and $z \in \R$. Using  (\ref{one}) and  (\ref{two}), then the triangle inequality and the smoothness of $F$,
\begin{align*}
2\|F_{n}-F\|_{\infty}+\|F_{2n}-F\|_{\infty} & \geq\left\Vert T_{L(2n)/L(n)}F_{2n}-F\right\Vert _{\infty}+\|{}T_{L(2n)/L(n)}F_{2n}-T_{L(2n)/L(n)}F\|_{\infty}\\
 & \geq\left\Vert T_{L(2n)/L(n)}F-F\right\Vert _{\infty} \\
 & =\sup_{x \in\mathbb{R}}\left|F\left(\frac{L(n)}{L(2n)} x\right)-F(x)\right|  \\
  & \geq \left|F\left(\frac{L(n)}{L(2n)} z \right)-F(z)\right| \\
  & = \left| F'(z) \left( \frac{L(n)}{L(2n)} - 1 \right) z+ O \left( \frac{L(n)}{L(2n)} - 1 \right)^2 \right|  \\
  & \geq C |zF'(z)| \left|1- \frac{L(n)}{L(2n)}  \right| 
 \end{align*}
for all sufficiently large $n$, as claimed.
\end{proof}
A similar argument shows slow convergence of the density functions.
\begin{theorem}
Let $X_{1},X_{2},\ldots$ be a sequence of independent
random variables with a common density. 
Denote by $S_{n}$ the sum $S_{n}=\sum_{i=1}^{n}X_{i}$.
Denote by $\rho_n$ the density
of $S_{n}/ ( n^{1/\alpha}L(n)),$ where \textup{$0<\alpha\leq2$
}and $L$ is slowly varying. Assume that
there is a symmetric and $\alpha$-stable random variable $Z$ with
density $\rho$ 
such that $\left\Vert \rho_n- \rho \right\Vert _{\infty}\rightarrow0$
as $n\rightarrow\infty$. Then for all $C<1$ and $z \in \R$,\textup{
\begin{equation}
\label{eq:low_bound_rho}
2^{(\alpha+1)/\alpha}\left\Vert \rho_n-\rho\right\Vert _{\infty}+\left\Vert \rho_{2n}-\rho \right\Vert _{\infty}\geq C \left|z \rho'(z) + \rho(z) \right|\left|1-\frac{L(n)}{L(2n)}\right|
\end{equation}
}for all sufficiently large $n$. Hence, provided $L(n)\neq L(2n)$ for all sufficiently large $n$,
\begin{equation}
\label{limsup_rho}
\limsup_{n\rightarrow\infty}\frac{\left\Vert \rho_n-\rho\right\Vert _{\infty}}{\left|1-L(n)/L(2n)\right|}>0 ~~\mbox{or} ~~
\limsup_{n\rightarrow\infty}\frac{\left\Vert \rho_{2n}-\rho\right\Vert _{\infty}}{\left|1-L(n)/L(2n)\right|}>0.
\end{equation}

\end{theorem}
\begin{proof}
From the $\alpha$-stability of $Z$, we have that $ \rho =~\! \tau_{2^{-1/\alpha}}\rho \ast \tau_{2^{-1/\alpha}}\rho.$
As earlier, we write
$\tilde{S}_{n}=\sum_{i=n+1}^{2n}X_{i}$.
Since $\tau_{L(2n)/L(n)}\rho_{2n}$ is the density of 
\[
\frac{L(2n)}{L(n)}\frac{S_{2n}}{\left(2n\right)^{1/\alpha}L(2n)}=2^{-1/\alpha} ~\! \frac{S_{n}}{n^{1/\alpha}L(n)}+2^{-1/\alpha} ~\! \frac{\tilde{S}_{n}}{n^{1/\alpha}L(n)},
\]
we have that
\begin{align}
\nonumber
\left\Vert \tau_{L(2n)/L(n)}\rho_{2n}-\rho\right\Vert _{\infty} & =\left\Vert \tau_{2^{-1/\alpha}}\rho_n\ast{}\tau_{2^{-1/\alpha}}\rho_n-{}\tau_{2^{-1/\alpha}}\rho\ast{}\tau_{2^{-1/\alpha}}\rho\right\Vert _{\infty}\\
\nonumber
 & =2^{1/\alpha}\left\Vert \rho_n\ast \rho_n-\rho\ast \rho\right\Vert _{\infty}\\
 \nonumber
 & \leq2^{1/\alpha}\left(\left\Vert \rho_n\ast \rho_n-\rho_n\ast \rho\right\Vert _{\infty}+\left\Vert \rho_n\ast \rho-\rho\ast \rho\right\Vert _{\infty}\right)\\
 \nonumber
 & \leq2^{1/\alpha}\left(\left\Vert \rho_n-\rho\right\Vert _{\infty}+\left\Vert \rho_n-\rho\right\Vert _{\infty}\right)\\
 & =2^{(\alpha+1)/\alpha}\left\Vert \rho_n-\rho\right\Vert _{\infty},
 \label{one_rho}
\end{align}
since $\int \rho=\int \rho_n=1.$ 
On the other hand,
\begin{equation}
\label{two_rho}
\left\Vert \tau_{L(2n)/L(n)}\rho_{2n}-\tau_{L(2n)/L(n)}\rho\right\Vert _{\infty}=\left|\frac{L(n)}{L(2n)}\right|\left\Vert \rho_{2n}-\rho\right\Vert _{\infty}.
\end{equation}
Let $z \in \R$. Using (\ref{one_rho}) and (\ref{two_rho}), then the triangle inequality, and then 
the smoothness of $\rho$, we have
\begin{align}
\label{first_line}
& ~~ 2^{(\alpha+1)/\alpha}\left\Vert \rho_n-\rho\right\Vert _{\infty}+\left|\frac{L(n)}{L(2n)}\right|\left\Vert \rho_{2n}-\rho\right\Vert _{\infty}   \\ \nonumber
& \geq\left\Vert \tau_{L(2n)/L(n)}\rho_{2n}-\rho\right\Vert _{\infty}+\left\Vert \tau_{L(2n)/L(n)}\rho_{2n}-\tau_{L(2n)/L(n)}\rho\right\Vert _{\infty}\\ \nonumber
 & \geq\left\Vert \tau_{L(2n)/L(n)}\rho-\rho\right\Vert _{\infty}\\ \nonumber
 & =\sup_{x\in\mathbb{R}}\left|\frac{L(n)}{L(2n)}\rho\left(\frac{L(n)}{L(2n)}x\right)-\rho(x)\right| \\ \nonumber
  & \geq \left|\frac{L(n)}{L(2n)}\rho\left(\frac{L(n)}{L(2n)}z\right)-\rho(z)\right| \\ \nonumber
 & = \left| \frac{L(n)}{L(2n)} \rho \left( z + \left( \frac{L(n)}{L(2n)} - 1 \right) z \right) - \rho(z) \right| \\ \nonumber
 & = \left| \left( \frac{L(n)}{L(2n)} -1 \right) \rho(z) + \frac{L(n)}{L(2n)} \rho'(z) 
 \left( \frac{L(n)}{L(2n)} -1 \right)  z + O \left( \frac{L(n)}{L(2n)} -1 \right)^2 \right| \\ \nonumber
 & = \left| z \rho'(z) + \rho(z) \right| \left| 1-\frac{L(n)}{L(2n)}  \right| + 
 O  \left( 1-\frac{L(n)}{L(2n)}  \right)^2. \nonumber
 \end{align}
This implies the estimate (\ref{eq:low_bound_rho}); note that the factor of $|L(n)/L(2n)|$ appearing
in (\ref{first_line}) becomes irrelevant in the limit as $n \rightarrow 1$, as its limit equals 1.  This in turn implies 
eq.\ (\ref{limsup_rho}) because $z \rho'(z) + \rho(z)$ is not identically zero, as
$\rho$ is not a constant multiple of $	1/z$.
\end{proof}

As discussed earlier, scaling by $\sqrt{n \log n}$ is  frequently encountered in the literature. 
As an example of an application of
the above theorems, we state the bounds on rates of convergence for
a generalization of this scaling.
\begin{corollary*}
Let $X_{1},X_{2},\ldots$ be a sequence of independent, identically
distributed, random variables. 
Denote by $S_{n}$ the sum $S_{n}=\sum_{i=1}^{n}X_{i}$, and by $F_n$ and $\rho_n$ the distribution and density, respectively, of $S_n/\left( n^{1/\alpha}\left(\log n\right)^{r} \right)$, for some
$\alpha$ with $0 < \alpha \leq 2$ and $r>0$.
Let $F$ and $\rho$ be symmetric, $\alpha$-stable distribution and density functions. Then
\[
\limsup_{n\rightarrow\infty }\log n\|F_n -F\|_\infty >0~\!
\]
and
\[
\limsup_{n\rightarrow\infty} \log n\|\rho_n -\rho\|_\infty >0.
\]
\end{corollary*}
\begin{proof}
This follows immediately from the two theorems above, since
\[
\lim_{n\rightarrow\infty}\log n\left(1-\left(\frac{\log n}{\log(2n)}\right)^{r}\right)=r\log2.
\]
\end{proof}

The following proposition makes clear why we characterize the rates of convergence in the two theorems as being ``at best logarithmic." 

\begin{proposition*}
Let $L$ be a slowly varying function such that $L(x)\rightarrow\infty$ or $L(x)\rightarrow0$
as $x\rightarrow\infty$. Then for any $\epsilon>0$, 
\begin{equation}
\limsup_{n\rightarrow\infty}(\log n)^{1+\epsilon}\left|1-\frac{L(n)}{L(2n)}\right|=\infty. \label{eq:prop2}
\end{equation}
\end{proposition*}
\begin{proof}
Assume first that $L(x)\rightarrow\infty$ as $x\rightarrow\infty$. Then for all integers $n\geq1$, $K\geq1$, 
\begin{equation}
\sum_{k=0}^{K-1}\log\frac{L\left(2^{k}n\right)}{L\left(2^{k+1}n\right)}=\log L\left(n\right)-\log L\left(2^{K}n\right).\label{eq:really?}
\end{equation}
Since, for fixed $n$, $\log L\left(2^{K}n\right)$$\rightarrow\infty$ as $K\rightarrow\infty$, the sum on the left-hand side of (\ref{eq:really?}) goes to negative infinity. Hence, the series tends to infinity absolutely, and we have,
for all $n\geq1$, 
\[
\sum_{k=0}^{\infty}\left| \log\frac{L\left(2^{k}n\right)}{L\left(2^{k+1}n\right)}\right|=\infty.
\]
 Since $L$ is slowly varying, $L\left(2^{k}n\right)/L\left(2^{k+1}n\right)\rightarrow1$
as $k\rightarrow\infty$. It follows that
\begin{equation}
\sum_{k=0}^{\infty}\left|1-\frac{L\left(2^{k}n\right)}{L\left(2^{k+1}n\right)}\right|=\infty.\label{eq:infinitesum}
\end{equation}
Now assume that eq.\ (\ref{eq:prop2}) is false. Then for some $\epsilon>0$, $C>0$, and all $n\geq2$,
\[
\left|1-\frac{L(n)}{L(2n)}\right|\leq\frac{C}{(\log n)^{1+\epsilon}}.
\]
 Then
\[
\sum_{k=0}^{\infty}\left|1-\frac{L\left(2^{k}n\right)}{L\left(2^{k+1}n\right)}\right|\leq\sum_{k=0}^{\infty}\frac{C}{\left(\log\left(2^{k}n\right)\right)^{1+\epsilon}}=\sum_{k=0}^{\infty}\frac{C}{(k\log2+\log n)^{1+\epsilon}}<\infty,
\]
which contradicts  eq.\  (\ref{eq:infinitesum}) and thereby  establishes   eq.\   (\ref{eq:prop2}) for $L(x)\rightarrow\infty$. 

Now assume that $L(x)\rightarrow0$ as $x\rightarrow\infty$. Set $\calL(x)=1/L(x)$. Then $\mathcal{L}$ is slowly varying and $\mathcal{L}(x)\rightarrow\infty$ as $x\rightarrow\infty$. Hence, 
\begin{align}
\nonumber
\infty= &
\limsup_{n\rightarrow\infty}(\log n)^{1+\epsilon}\left|1-\frac{\calL(n)}{\calL(2n)}\right|\\
\nonumber
= & \limsup_{n\rightarrow\infty}(\log n)^{1+\epsilon}\left|1-\frac{\calL(2n)}{\calL(n)}\right|\\
\nonumber
=&  \limsup_{n\rightarrow\infty}(\log n)^{1+\epsilon}\left|1-\frac{L(n)}{L(2n)}\right|
\end{align}
The second equality follows from
$$
\lim_{n \rightarrow \infty} \frac{\calL(2n)/\calL(n)-1}{\calL(n)/\calL(2n)-1} = 
\lim_{n \rightarrow \infty} \left( - \frac{\calL(2n)}{\calL(n)}  \right) = -1. 
$$
\end{proof}

Note that the hypotheses of the proposition don't imply  (\ref{eq:prop2}) for $\epsilon=0$.
For example, in the canonical case discussed above, $L(n)=\sqrt{\log n},$
we find 
\[
\lim_{n\rightarrow\infty}\log n\left|1-\frac{L(n)}{L(2n)}\right|=\lim_{n\rightarrow\infty}\log n\left(1-\sqrt{\frac{\log(n)}{\log(2n)}}\right)=\frac{\log2}{2}.
\]

As discussed earlier, our theorems are only best-case estimates. The precise rates of convergence depend on the slowly varying functions in the scalings. As an example, fix a constant $A>0$ and consider a sequence $X_{1},X_{2},\ldots$ of identically distributed random variables
which have a common smooth, symmetric density function, which for large $|x|$ 
is equal to 
$A/(2|x|^3)$. 
Denote by $q_n$ the density of
\[
\frac{\sum_{i=1}^{n}X_{i}}{h(n)},
\]
where  the function $h$ is defined by $h^2 = n \log h$ and $\lim_{n \rightarrow \infty} h(n) = \infty$, and by $ \rho$ the Gaussian density with mean $0$ and variance $A$. Kuske and Keller \cite{keller2001rate} proved that $\| q_n - \rho \|_\infty=O\left(1/\log n\right)$.
However, consider the natural scaling $\sqrt{(n\log n) /2}$ and denote by $\rho_n$ the density of 
\[
\frac{\sum_{i=1}^{n}X_{i}}{\sqrt{(n\log n)/2}}.
\]
Then it can be shown that the convergence of the scaled partial sums deteriorates slightly, namely we have
\[
c \frac{\log \log n}{\log n} \leq \| \rho_n-\rho \|_\infty  \leq 
C \frac{\log \log n}{\log n}
\]
with $0 < c < C < \infty$.
This follows from the fact that $h(n) = (1+c_n)\sqrt{(n \log n)/2}$, with $c_n$ asymptotic to $\log \log n/(2\log n)$. We note that N\'andori \cite[Theorem 1]{nandori2011recurrence} proved an analogous result for 
a random walk on a lattice, with the same rate of convergence.

\vskip 10pt
\noindent
{\bf Acknowledgments.} We are grateful to G\'erard Ben Arous for very helpful suggestions. The second author thanks the Courant Institute for hosting him during the 2017/18 academic year as a Visiting Scholar.


\begin{thebibliography}{10}

\bibitem{balint2006limit}
P.~B{\'a}lint and S.~Gou{\"e}zel.
\newblock Limit theorems in the stadium billiard.
\newblock {\em Comm. Math. Phys.}, 263(2):461--512, 2006.

\bibitem{borgers1992diffusion}
C.~B{\"o}rgers, C.~Greengard, and E.~Thomann.
\newblock The diffusion limit of free molecular flow in thin plane channels.
\newblock {\em SIAM J. Appl. Math}, 52(4):1057--1075, 1992.

\bibitem{boutsikas2001compound}
M.~V. Boutsikas and M.~V. Koutras.
\newblock Compound {P}oisson approximation for sums of dependent random
  variables,
\newblock in C.~A. Charalambides, M.~V. Koutras, and N.~Balakrishnan (Eds.),
 {\sl Probability and Statistical Models with Applications: A volume in honour of Prof. T.~Cacoullos}, 2001, pp. 63--86.

\bibitem{ChristophWolf}
G.~Christoph and W.~Wolf.
\newblock Convergence theorems with a stable limit law. 
\newblock Akademie Verlag, Berlin, 1992.


\bibitem{chumley2016diffusivity}
T.~Chumley, R.~Feres, and H.-K. Zhang.
\newblock Diffusivity in multiple scattering systems.
\newblock {\em Trans. Amer. Math. Soc.}, 368(1):109--148, 2016.

\bibitem{cristadoro2014measuring}
G.~Cristadoro, T.~Gilbert, M.~Lenci, and D.~P. Sanders.
\newblock Measuring logarithmic corrections to normal diffusion in
  infinite-horizon billiards.
\newblock {\em Phys. Rev. E}, 90(2):022106, 2014.

\bibitem{dettmann2014diffusion}
C.~P. Dettmann.
\newblock Diffusion in the {L}orentz gas.
\newblock {\em Commun. Theor. Physics}, 62(4):521--540, 2014.

\bibitem{gouezel2004central}
S.~Gou{\"e}zel.
\newblock Central limit theorem and stable laws for intermittent maps.
\newblock {\em Prob. Theory and Related Fields}, 128(1):82--122, 2004.

\bibitem{jimenez1996algebraic}
J.~Jim\'enez.
\newblock Algebraic probability density tails in decaying isotropic
  two-dimensional turbulence.
\newblock {\em J. Fluid Mech.}, 313:223--240, 1996.

\bibitem{juozulynas1998some}
A.~Juozulynas and V.~Paulauskas.
\newblock Some remarks on the rate of convergence to stable laws.
\newblock {\em Lith. Math. J}, 38(4):335--347, 1998.

\bibitem{keller2001rate}
R.~Kuske and J.~B. Keller.
\newblock Rate of convergence to a stable law.
\newblock {\em SIAM J. Appl. Math}, 61(4):1308--1323, 2001.

\bibitem{nandori2011recurrence}
P.~N{\'a}ndori.
\newblock Recurrence properties of a special type of heavy-tailed random walk.
\newblock {\em J. Stat. Phys.}, 142(2):342--355, 2011.

\bibitem{nolan2010bibliography}
J.~P. Nolan.
\newblock Bibliography on stable distributions, processes and related topics.
\newblock {\em {\rm http://fs2.american.edu/jpnolan/www/stable/stable.html}},
  2017.

\bibitem{Rachevetal}
S.~T.~Rachev, et al.
\newblock The methods of distances in the theory of probability and statistics. 
\newblock Springer Science \& Business Media, 2013.

\bibitem{Zolo}
V.~M. Zolotarev.
One-dimensional stable distributions. Trans. Math. Monogr. 65, AMS, 1986 (Russian, Nauka, 1983).

\end{thebibliography}

\end{document}